\documentclass[11pt,twoside,reqno]{amsart}
\allowdisplaybreaks
\usepackage{amsmath,amstext,amssymb,epsfig,multicol,enumerate}
\usepackage{graphicx}
\usepackage{mathrsfs}
\calclayout
\makeatletter
\renewcommand{\@seccntformat}[1]{\bf\csname the#1\endcsname.}
\renewcommand{\section}{\@startsection{section}{1}
	\z@{.7\linespacing\@plus\linespacing}{.5\linespacing}
	{\normalfont\upshape\bfseries\centering}}
\renewcommand{\@biblabel}[1]{\@ifnotempty{#1}{#1.}}
\makeatother
\theoremstyle{plain}
\newtheorem{thm}{Theorem}[section]
\newtheorem{lem}[thm]{Lemma}
\newtheorem{prop}[thm]{Proposition}
\newtheorem{cor}[thm]{Corollary}

\theoremstyle{definition}

\newtheorem{defn}[thm]{Definition}
\newtheorem{rem}{Remark}[section]

\usepackage{cancel}
\usepackage[parfill]{parskip}
\usepackage[german]{varioref}
\usepackage[all]{xy}
\usepackage{color}

\def \>{\succ}
\def \<{\prec}

\DeclareMathOperator{\QDer}{QDer}

\begin{document}	
\title[Basdouri Imed\textsuperscript{1}, Jean Lerbet\textsuperscript{2}, Bouzid Mosbahi\textsuperscript{3}]{Quasi-centroids and Quasi-Derivations of low-dimensional Zinbiel algebras }
	\author{Basdouri Imed\textsuperscript{1}, Jean Lerbet\textsuperscript{2}, Bouzid Mosbahi\textsuperscript{3}}
 \address{\textsuperscript{1}Department of Mathematics, Faculty of Sciences, University of Gafsa, Gafsa, Tunisia}
 \address{\textsuperscript{2}Laboratoire de Mathématiques et Modélisation d’Évry (UMR 8071) Université d’Évry Val d’Essonne I.B.G.B.I., 23 Bd. de France, 91037 Évry Cedex, France}
\address{\textsuperscript{3}Department of Mathematics, Faculty of Sciences, University of Sfax, Sfax, Tunisia}

 \email{\textsuperscript{1}basdourimed@yahoo.fr}
 \email{\textsuperscript{2}jean.lerbet@ufrst.univ-evry.fr}
\email{\textsuperscript{3}mosbahi.bouzid.etud@fss.usf.tn}
	

	

	\date{\today}
\begin{abstract}  In this paper, we introduce the concepts of quasi-centroid and quasi-derivation for Zinbiel algebras. Utilizing the classification results of Zinbiel algebras established previously, we describe the quasi-centroids and quasi-derivations of low-dimensional Zinbiel algebras. Additionally, we explore certain properties of quasi-centroids in the context of Zinbiel algebras and employ these properties to classify algebras with so-called small quasi-centroids. This description of quasi-derivations allows us to identify a significant subclass of Zinbiel algebras characterized as quasi-characteristically nilpotent.
\end{abstract}

\maketitle \section{Introduction}\label{introduction}
Non-associative algebras have been widely studied for their theoretical importance and broad applications in fields like physics, engineering, and applied mathematics. Within this area, Zinbiel algebras hold a distinct position. Introduced by J.-L. Loday in 1995, Zinbiel algebras serve as the Koszul duals of Leibniz algebras in \cite{1,2,3}. The term \textit{Zinbiel}, suggested by J.M. Lemaire, reflects this relationship, as it is ``Leibniz'' spelled backward, indicating their association with commutative associative algebras, pre-Lie algebras, and dendriform algebras in \cite{4}.
Zinbiel algebras are defined by a specific right-symmetric identity:
\begin{align*}
(p\cdot q)\cdot r &= p\cdot (q \cdot r)+p\cdot(r\cdot q)
\end{align*}
and naturally connect to dendriform algebras, appearing in contexts such as algebraic combinatorics, noncommutative geometry, and the study of operads. These algebras generalize commutative structures in a way similar to how Leibniz algebras extend Lie algebras. The study of Zinbiel algebras has evolved to include various applications and links to fields such as classification, derivatins, centroids, homological algebra, deformation theory, and quantum field theory in \cite{5,6,7,8,9}.

This paper focuses on defining and describing quasi-centroids and quasi-derivations of Zinbiel algebras. Quasi-centroids and quasi-derivations are derived from investigations of low-dimensional Zinbiel algebras, where these structures play a key role in classification and algebraic applications. Our analysis draws on existing classifications of complex Zinbiel algebras in dimensions two, three, and four, referenced in works such as \cite{10,11,12}.

The structure of this paper is as follows: Section 1 provides an introduction to Zinbiel algebras along with a review of key results. Section 2 presents the basic definitions and preliminary concepts essential for understanding Zinbiel algebras. Section 3 focuses on the properties of quasi-centroids and quasi-derivations. Sections 4 and 5 describe and compute quasi-centroids and quasi-derivations in finite-dimensional Zinbiel algebras. We utilize classification results for low-dimensional complex Zinbiel algebras from \cite{13} and illustrate our methods and computations with examples generated using Mathematica software.

\section{Prelimieries}
This section introduces essential definitions and foundational concepts that will be relevant for the study of Zinbiel algebras.

\begin{defn}
An algebra  $Z$ over a field $\mathbb{K}$ is a vector space over $\mathbb{K}$ equipped with a bilinear map
\begin{align*}
\xi: Z \times Z \to Z,
\end{align*}
such that
\begin{align*}
\xi(\alpha p + \beta q, r) &= \alpha \xi(p, r) + \beta \xi(q, r),\\
\xi(r, \alpha p + \beta q) &= \alpha \xi(r, p) + \beta \xi(r, q),
\end{align*}
where  $p, q, r \in Z$  and  $\alpha, \beta \in \mathbb{K}$.
\end{defn}

\begin{defn}
A Zinbiel algebra  $Z$  over a field  $\mathbb{K}$ is an algebra satisfying the following condition:
\begin{align*}
(p\cdot q)\cdot r &= p\cdot (q \cdot r)+p\cdot(r\cdot q)
\end{align*}
for all \quad $p, q, r \in Z$.
\end{defn}

Under the conditions of a new product presented as $p\cdot q + q\cdot p$.
It is possible to obtain a structure of the associative commutative algebra on Z. The
definition of a sequence for a provided Zinbiel algebra Z can be given in the
following way:
$$Z^{1} = Z,\quad Z^{t+1} = Z\cdot Z^{t},\quad t \geq 1$$

\begin{defn}
Let  $Z$ and $Z'$ be two Zinbiel algebras over a field $\mathbb{K}$. A linear mapping  $\psi : Z \rightarrow Z_1$ is a \textit{homomorphism} if
\begin{align*}
\psi(p \cdot q) = \psi(p) \ast \psi(q), \;\text{for all}\; p, q \in Z.
\end{align*}
\end{defn}

\begin{defn}
Let  $Z$ be a Zinbiel algebra over a field $\mathbb{K}$. A linear mapping  $\psi : Z \rightarrow Z$  is an \textit{endomorphism} if
\begin{align*}
\psi(p \cdot q) &= \psi(p) \cdot \psi(q), \;\text{for all}\; p, q \in Z.
\end{align*}
\end{defn}

We denote $End_{\mathbb{K}}(Z)$ as the set of all endomorphisms of \( Z \).

\begin{defn}
A linear transformation  $d$  of a Zinbiel algebra  $Z$  is called a \textit{derivation} if for any $p, q \in Z$,
\begin{align*}
d(p \cdot q) &= d(p) \cdot q + p \cdot d(q).
\end{align*}
\end{defn}

The space of all derivations of the algebra $Z$, equipped with the multiplication defined as the commutator, forms a Lie algebra which is denoted by $Der(Z)$.

\begin{prop}
Let $Z$ be a Zinbiel algebra. Then $Der(Z)$ is a Lie algebra with respect to the bracket $[d_1, d_2] = d_1 \circ d_2 - d_2 \circ d_1$.
\end{prop}

\begin{proof}
For $d_1, d_2 \in Der(Z)$, we need to verify that the bracket $[d_1, d_2] = d_1 \circ d_2 - d_2 \circ d_1$ is also a derivation on $Z$, i.e., that it satisfies the derivation property:
\begin{align*}
[d_1, d_2](p \cdot q) &= [d_1, d_2](p) \cdot q + p \cdot [d_1, d_2](q).
\end{align*}

Calculating $[d_1, d_2](p \cdot q)$ :

\begin{align*}
[d_1, d_2](p \cdot q) &= (d_1 \circ d_2)(p \cdot q) - (d_2 \circ d_1)(p \cdot q) \\
&= d_1(d_2(p \cdot q)) - d_2(d_1(p \cdot q)).
\end{align*}
Using the derivation property of  $d_2$  on  $p \cdot q$ , we get
\begin{align*}
d_2(p \cdot q) &= d_2(p) \cdot q + p \cdot d_2(q).
\end{align*}
Thus,
\begin{align*}
d_1(d_2(p \cdot q)) &= d_1(d_2(p) \cdot q + p \cdot d_2(q)) \\
&= d_1(d_2(p)) \cdot q + d_2(p) \cdot d_1(q) + d_1(p) \cdot d_2(q) + p \cdot d_1(d_2(q)),
\end{align*}
and similarly,
\begin{align*}
d_2(d_1(p \cdot q)) &= d_2(d_1(p) \cdot q + p \cdot d_1(q)) \\
&= d_2(d_1(p)) \cdot q + d_1(p) \cdot d_2(q) + d_2(p) \cdot d_1(q) + p \cdot d_2(d_1(q)).
\end{align*}
Subtracting these two expressions, we find
\begin{align*}
[d_1, d_2](p \cdot q) &= \big(d_1(d_2(p)) - d_2(d_1(p))\big) \cdot q + p \cdot \big(d_1(d_2(q)) - d_2(d_1(q))\big) \\
&= [d_1, d_2](p) \cdot q + p \cdot [d_1, d_2](q).
\end{align*}
Thus, $[d_1, d_2]$ satisfies the derivation property and hence  $[d_1, d_2] \in Der(Z)$.

Next, to conclude that $Der(Z)$ is a Lie algebra, we need to check two properties of the bracket $[\cdot, \cdot]$:

1. \textit{Antisymmetry}: For any $d_1, d_2 \in Der(Z)$,
   \begin{align*}
   [d_1, d_2] &= -[d_2, d_1],
   \end{align*}
   which follows directly from the definition:
   \begin{align*}
   [d_1, d_2] &= d_1 \circ d_2 - d_2 \circ d_1 = -(d_2 \circ d_1 - d_1 \circ d_2) = -[d_2, d_1].
  \end{align*}

2. \textit{Jacobi identity}: For any $d_1, d_2, d_3 \in Der(Z)$, we must show that
\begin{align*}
   [d_1, [d_2, d_3]] + [d_2, [d_3, d_1]] + [d_3, [d_1, d_2]] &= 0.
\end{align*}
   This follows from the fact that $[\cdot, \cdot]$ is the commutator of linear maps, which always satisfies the Jacobi identity.

Since both antisymmetry and Jacobi identity hold, we conclude that $Der(Z)$ is a Lie algebra under the bracket $[d_1, d_2] = d_1 \circ d_2 - d_2 \circ d_1$.
\end{proof}

\begin{prop}
Let $(Z, \cdot)$ be a Zinbiel algebra and $d \in \text{Hom}(Z)$. Then the following conditions hold:
\begin{enumerate}
    \item $d \in \text{Der}(Z)$
    \item $[d, L_{p}] = L_{d(p)}$
    \item $[d, R_{p}] = R_{d(p)}$
\end{enumerate}
\end{prop}

\begin{proof}
Let $(Z, \cdot)$ be a Zinbiel algebra and $d \in \text{Hom}(Z)$. We will show each part of the proposition.

\begin{enumerate}
    \item  $d \in \text{Der}(Z)$.
    \begin{align*}
    d(p \cdot q) = d(p) \cdot q + p \cdot d(q) \quad \text{for all } p, q \in Z.
    \end{align*}
    Since $d$ is in $\text{Hom}(Z)$, it is a linear map, and we assume it satisfies this condition as given.
    \item $[d, L_p] = L_{d(p)}$.

    Here, we need to verify that the commutator of $d$ and $L_p$ is equal to $L_{d(p)}$, where $L_p$ is the left multiplication operator defined by $L_p(q) = p \cdot q$ for all $q \in Z$.

    The commutator $[d, L_p]$ is given by
\begin{align*}
    [d, L_p](q) &= d(L_p(q)) - L_p(d(q)) = d(p \cdot q) - p \cdot d(q).
 \end{align*}
    By the derivation property of $d$, we have
\begin{align*}
    d(p \cdot q) &= d(p) \cdot q + p \cdot d(q).
\end{align*}
    Substituting this into the equation for $[d, L_p](q)$, we get
    \begin{align*}
    [d, L_p](q) &= (d(p) \cdot q + p \cdot d(q)) - p \cdot d(q) = d(p) \cdot q.
    \end{align*}
    Therefore,
    \begin{align*}
    [d, L_p](q) &= L_{d(p)}(q),
    \end{align*}
    which confirms that $[d, L_p] = L_{d(p)}$.

    \item $[d, R_p] = R_{d(p)}$.

    In this part, we need to show that $[d, R_p] = R_{d(p)}$, where $R_p$ is the right multiplication operator defined by $R_p(q) = q \cdot p$ for all $q \in Z$.

    The commuter $[d, R_p]$ is given by
\begin{align*}
    [d, R_p](q) &= d(R_p(q)) - R_p(d(q)) = d(q \cdot p) - d(q) \cdot p.
  \end{align*}
    Using the derivation property of $d$ again, we have
\begin{align*}
    d(q \cdot p) &= d(q) \cdot p + q \cdot d(p).
   \end{align*}
    Substituting this into the equation for $[d, R_p](q)$, we get
 \begin{align*}
    [d, R_p](q) &= (d(q) \cdot p + q \cdot d(p)) - d(q) \cdot p = q \cdot d(p).
    \end{align*}
    Thus,
   \begin{align*}
    [d, R_p](q) &= R_{d(p)}(q),
    \end{align*}
    which confirms that $[d, R_p] = R_{d(p)}$.
\end{enumerate}
\end{proof}

\begin{lem}
The sets $R(Z) = \{ R_p \mid p \in Z$ \}  and  $L(Z) = \{ L_p \mid p \in Z \}$ are subalgebras of the Zinbiel algebra $Der(Z)$.
\end{lem}

\begin{proof}
For \( R(Z) = \{ R_p \mid p \in Z \} \), consider any  $p, q \in Z$. By definition of the operators $R_p$ and  $R_q$ , we have
\begin{align*}
R_{p \cdot q} &= R_q R_p.
\end{align*}
This shows that  $R(A)$ is closed under the composition of the operators $R_p$  and  $R_q$, satisfying the closure property.

Similarly, for  $L(Z) = \{ L_p \mid p \in Z \}$, we consider any $p, q \in Z$. By the definition of the left multiplication operators  $L_p$  and  $L_q$, we have
\begin{align*}
L_{p \cdot q} &= L_p L_q.
\end{align*}
This confirms that  $L(Z)$ is also closed under the composition, ensuring it satisfies the closure property as well.

Since both  $R(Z)$ and $L(Z$) are closed under the Zinbiel product and contain the necessary identity elements (if required by the structure of $Der(Z)$), they are indeed subalgebras of the Zinbiel algebra $Der(Z)$.
\end{proof}

\section{Proprieties of quasi-centroids and quasi-derivations}

\begin{defn}
The \textit{centroid} of a Zinbiel algebra $(Z, \cdot)$ over a field $\mathbb{K}$ is defined by
\begin{align*}
\Gamma(Z) &= \left\{\phi \in \text{End}_{\mathbb{K}}(Z) \mid \phi(p \cdot q) = \phi(p) \cdot q = p \cdot \phi(q), \; \forall \; p, q \in Z \right\}.
\end{align*}
\end{defn}

\begin{rem}
Let $I$ be a nonempty subset of $Z$. The \textit{centralizer} of $I$ in $Z$ is the set
\begin{align*}
C_{Z}(I) &= \{ p \in Z \mid p \cdot I = I \cdot p = 0 \}.
\end{align*}
If $I$ is an ideal of $Z$, then $C_{Z}(I)$ is also an ideal of $Z$. In particular, if $I = Z$, we write $C_{Z}(I) = C(Z)$.
\end{rem}

\begin{defn}
Let $Z$ be a Zinbiel algebra. The quasi-centroid of $Z$, denoted by $Q\Gamma(Z)$, is defined as the set
\[ Q\Gamma(Z) = \left\{ \phi \in \text{End}_{\mathbb{K}}(Z) \mid \phi(p) \cdot q = p \cdot \phi(q), \ \forall\; p, q \in Z \right\}. \]
\end{defn}

\begin{defn}
The central derivation of a Zinbiel algebra $Z$, denoted by $C_d(Z)$, is the subset of $\text{End}_{\mathbb{K}}(Z)$ given by
\[ C_d(Z) = \left\{ \phi \in \text{End}_{\mathbb{K}}(Z) \mid \phi(p) \cdot q = p \cdot \phi(q) = 0, \ \forall \; p, q \in Z \right\}. \]
\end{defn}

\begin{defn}
An element $d \in \text{End}_{\mathbb{K}}(Z)$ is called a quasi-derivation if there exists another map $d' \in \text{End}_{\mathbb{K}}(Z)$ such that
\[ d(p) \cdot q + p \cdot d(q) = d'(p \cdot q), \ \forall \; p, q \in Z. \]
\end{defn}

\begin{defn}
A Zinbiel algebra $Z$ is called nilpotent if there exists a positive integer $s \in \mathbb{N}$ such that $Z^s = \{0\}$, where $Z^k$ denotes the $k$-th power of $Z$ under the Zinbiel product. The smallest such integer $s$ is referred to as the nil-index of $Z$.
\end{defn}

\begin{defn}
If the quasi-derivations of $Z$ form a nilpotent algebra, we say that $Z$ is characteristically nilpotent.
\end{defn}

\begin{defn}
A Zinbiel algebra $Z$ is called indecomposable if it cannot be expressed as a direct sum of two ideals. Otherwise, it is decomposable.
\end{defn}

\begin{defn}
Let $Z$ be an indecomposable Zinbiel algebra. If the central derivations and scalars of $Z$ form a small subalgebra $L$, then we say that $L$ is small.
\end{defn}

\noindent Next, we analyze the interactions between quasi-derivations and quasi-centroids within Zinbiel algebras.

\begin{prop}
Let $I$ be an ideal of the Zinbiel algebra $Z$. Then $C_{Z}(I)$, the centralizer of $I$ in $Z$, is itself an ideal of $Z$.
\end{prop}

\begin{proof}
Assume $p \in C_Z(I)$ and $q \in Z$, then for any $r \in I$, we have
\[ (p \cdot q) \cdot r = p \cdot (q \cdot r) = 0 \]
and similarly,
\[ r \cdot (p \cdot q) = 0, \]
showing that $C_{Z}(I)$ is an ideal of $Z$.
\end{proof}

\begin{lem}
Suppose $Z = U \oplus V$ where $U$ and $V$ are ideals of $Z$. Then we have $C_Z(Z) = C_Z(U) \oplus C_Z(V)$.
\end{lem}

\begin{proof}
Since $C_Z(U) \cap C_Z(V) = 0$, for any $p \in C_Z(Z)$ where $p = p_1 + p_2$ with $p_1 \in C_Z(U)$ and $p_2 \in C_Z(V)$, we find that $p \in C_Z(Z)$. Therefore, $C_Z(U) \oplus C_Z(V) \subseteq C_Z(Z)$.
\end{proof}

\begin{lem}
Let  $Z$ be a Zinbiel algebra and  $W $ a subset of  $Z$.

\begin{enumerate}
    \item$Q\Gamma(Z) \cdot QDer(Z) \subseteq QDer(Z)$.
    \item$[Q\Gamma(Z), QDer(Z)] \subseteq QDer(Z)$.
    \item$[Q\Gamma(Z), Q\Gamma(Z)](Z) \subseteq C(Z)$  and  $[Q\Gamma(Z), Q\Gamma(Z)](Z^1) = 0$.
    \item The centralizer $C_Z(W)$  of  $W$ in $Z$ is invariant under  $Q\Gamma(Z)$.
\end{enumerate}
\end{lem}

\begin{proof}
The results follow directly from the definitions of quasi-derivation and quasi-centroid.
\end{proof}

\begin{prop}
Let  $Z$ be a Zinbiel algebra and  $Q\Gamma(Z)$  its quasi-centroid.

\begin{enumerate}
    \item $C_d(Z) = Q\Gamma(Z) \cap QDer(Z)$.
    \item If $d \in QDer(Z)$,  $\phi \in Q\Gamma(Z)$, then  $d \phi \in Q\Gamma(Z)$  if and only if $\phi d$ is a central derivation of $Z$ .
    \item If  $d \in QDer(Z)$, $\phi \in Q\Gamma(Z)$, then  $d \phi \in QDer(Z)$ if and only if $[d, \phi]$ is a central derivation of  $Z$.
\end{enumerate}
\end{prop}

\begin{proof}
\begin{enumerate}
    \item By definition, $C_d(Z)$ consists of elements in \( \mathrm{End}_{\mathbb{K}}(Z) \) that satisfy the central derivation condition, i.e., $\phi(p) \cdot q = p \cdot \phi(q) = 0$ for all $p, q \in Z$ . Elements in $Q\Gamma(Z)$ commute under the Zinbiel product in the specified way, while elements in $ QDer(Z)$ act as quasi-derivations. Thus, $ C_d(Z)$ is the intersection  $Q\Gamma(Z) \cap QDer(Z)$.

    \item Suppose $ d \in QDer(Z)$ and $\phi \in Q\Gamma(Z)$. Then $d \phi \in Q\Gamma(Z)$ if and only if  $d \phi(p) \cdot q = p \cdot d \phi(q) = 0$ for all $p, q \in Z$, which holds if and only if  $\phi d$  satisfies the conditions for a central derivation of $Z$.

    \item For  $d \in QDer(Z)$ and $\phi \in Q\Gamma(Z),  d \phi \in QDer(Z)$ if and only if  $[d, \phi]$ (the commutator) satisfies the conditions for a central derivation, specifically that  $[d, \phi](p) \cdot q = p \cdot [d, \phi](q) = 0$  for all $p, q \in Z$.
\end{enumerate}
\end{proof}

\begin{thm}\label{t5}
Let  $Z = Z_1 \oplus Z_2$, where  $Z_1$ and  $Z_2$ are ideals of  $Z$. Then
\begin{align*}
Q\Gamma(Z) &= Q\Gamma(Z_1) \oplus Q\Gamma(Z_2) \oplus C_1 \oplus C_2,
\end{align*}
where
\begin{align*}
C_i &= \{ \phi \in \text{Hom}(Z_i, Z_j) \mid \phi(Z_i) \subset C(Z_j), \; \phi(Z_1) = 0 \}, \quad 1 \leq i \neq j \leq 2.
\end{align*}
\end{thm}

\begin{cor}
 If  $Z = Z_1 \oplus Z_2 \oplus \cdots \oplus Z_m$ , where $Z_i$ for  $i = 1, 2, \dots, m$ are ideals of  $Z$, then
\begin{align*}
Q\Gamma(Z) &= Q\Gamma(Z_1) \oplus \cdots \oplus Q\Gamma(Z_m) \oplus \bigoplus_{i \neq j} C_{ij},
\end{align*}
where  $C_{ij} = \{ \phi \in \mathrm{Hom}_K(Z_i, Z_j) \mid \phi(Z_i) \subseteq C(Z_j), \, \phi(Z_i) = 0 \, \text{if} \, i = j \}$, for $1 \leq i, j \leq m$.
\end{cor}

Under the result in Theorem \ref{t5}, we give the following definition.
\begin{defn}
Let Z be an indecomposable Zinbiel algebra. We will say $Q\Gamma(Z)$
is small if $Q\Gamma(Z)$ is generated by central derivations and the scalars. The quasi-centroid
of a decomposable Zinbiel algebra will be called small if the quasi-centroid of each
indecomposable factor is small.
\end{defn}

\section{Quasi-Centroids of Low-Dimensional Zinbiel Algebras}
This section provides the details of the quasi-centroid of Zinbiel algebras in dimensions two, three, and four over the complex field $\mathbb{C}$. We choose $\{e_1, e_2, \dots, e_n\}$ as the basis for a $n$-dimensional Zinbiel algebra $Z$. The product of the basis elements is defined as
\begin{align*}
e_ie_j &= \sum_{k=1}^n \gamma_{ij}^k e_k, \quad i, j = 1, 2, \dots, n.
\end{align*}

Since the quasi-centroid is a linear transformation of the vector space $Z$, an element $\phi \in Q\Gamma(Z)$ can be represented in a matrix form  $[a_{ij}]_{i,j=1,2,\dots,n}$, i.e.,
\begin{align*}
\phi(e_i) &= \sum_{j=1}^n a_{ij} e_j, \quad i = 1, 2, \dots, n.
\end{align*}
According to the definition of the quasi-centroid, the entries  $a_{ij} ,  j = 1, 2, \dots, n$, of the matrix $[a_{ij}]_{i,j=1,2,\dots,n}$ must satisfy
\begin{align}
\sum_{t=1}^n (\gamma_{it}^k a_{tj} - a_{it} \gamma_{tj}^k) &= 0.
\end{align}

The quasi-centroid of the Zinbiel algebra $Z$ under consideration can thus be described by solving this system of equations for $a_{ij}, i, j = 1, 2, \dots, n$, once the structure constants $\gamma_{ij}^k$ of $Z$ are known. Here we utilize classification results of two, three, and four-dimensional complex Zinbiel algebras from relevant literature.
to solve the system of equations above with respect to $a_{ij}$  $1 \leq i, j, t \leq n$.
which can be done by using computer software.
\begin{thm}\label{t1}
Any $2$-dimensional Zinbiel algebra Z isomorphic to
one of the following nonisomorphic Zinbiel algebras $Z_2^{1}:\;e_1e_1=e_2$.
\end{thm}

\begin{thm}
The quasi-centroids of the complex Zinbiel algebras of $2$ dimensions have the
following form:
\end{thm}
\begin{center}
\begin{tabular}{cccc}
\hline
Algebra & Q$\Gamma$(Z) & Dim &Type\;of\; Q$\Gamma$(Z)\\
\hline
$Z_2^{1}$
&$\left(\begin{array}{cccc}
a_{22}&0\\
a_{21}&a_{22}
\end{array}
\right)$
&2&
small\\
\end{tabular}
\end{center}

\begin{proof}
Consider $Z_{2}^{1}$ from theorem \ref{t1}. The structure constants of $Z_{2}^{1}$ are $\gamma_{11}^{2}=1$ and the others are zeros. Solving the system of equation (1), we have $a_{12}=0$ and $a_{22}=a_{11}.$
Therefore we obtain the centroids of $Z_{2}^{1}$ in matrix form as follows
$Q\Gamma(Z_{2}^{1})= \left\{\left(
\begin{array}{cc}
a_{22} & 0 \\
a_{21} & a_{22}
\end{array}
\right)| a_{21}, a_{21}\in\mathbb{C} \right\}.$
\end{proof}
\begin{cor}
The quasi-centroid of a $2$-dimensional Zinbiel algebra are small.\\
The dimensions of the quasi-centroids of $2$-dimensional Zinbiel algebras are two.
\end{cor}
\begin{thm}\label{t2}
Any $3$-dimensional Zinbiel algebra Z isomorphic to one of following non-isomorphic Zinbiel algebras\\
$Z_3^{1}:\;e_ie_j=0$\\
$Z_3^{2}:\;e_1e_1=e_3$\\
$Z_3^{3}:\;e_1e_1=e_3, e_2e_2=e_3$\\
$Z_3^{4}:\;e_1e_2=\frac{1}{2}e_3, e_2e_1=\frac{-1}{2}e_3$\\
$Z_3^{5}:\;e_2e_1=e_3$\\
$Z_3^{6}:\;e_1e_1=e_3,e_1e_2=e_3,e_2e_2=\lambda e_3, \quad \lambda  \neq 0 $\\
$Z_3^{7}:\;e_1e_1=e_2,e_1e_2=\frac{1}{2}e_3,e_2e_1=e_3$
\end{thm}
\begin{thm}
The Quasi-centroid of $3$-dimensional complex  Zinbiel algebras are given as follows:
\end{thm}
\begin{center}
\begin{tabular}{cccc}
\hline
Algebra & Q$\Gamma$(Z) & Dim &Type\;of\; Q$\Gamma$(Z)\\
\hline
$Z_3^{1}$
&$\left(\begin{array}{cccc}
a_{11}&a_{12}&a_{13}\\
a_{21}&a_{22}&a_{23}\\
a_{31}&a_{32}&a_{33}
\end{array}
\right)$
&9&
small\\
$Z_3^{2}$
&$\left(\begin{array}{cccc}
a_{33}&a_{12}&0\\
a_{21}&a_{22}&0\\
a_{31}&a_{32}&a_{33}
\end{array}
\right)$
&6&
not small\\
$Z_3^{3}$
&$\left(\begin{array}{cccc}
a_{33}&a_{12}&0\\
0&a_{33}&0\\
a_{31}&a_{32}&a_{33}
\end{array}
\right)$
&4&
small\\
$Z_3^{4}$
&$\left(\begin{array}{cccc}
a_{33}&0&0\\
0&a_{12}&0\\
a_{31}&a_{32}&a_{33}
\end{array}
\right)$
&4&
not small\\

$Z_3^{5}$
&$\left(\begin{array}{cccc}
a_{11}&a_{12}&a_{13}\\
0&a_{22}&a_{23}\\
a_{31}&a_{32}&a_{33}
\end{array}
\right)$
&8&
not small\\

$Z_3^{6}, \lambda \neq 0$
&$\left(\begin{array}{cccc}
a_{33}&g&0\\
0&a_{22}&0\\
a_{31}&a_{32}&a_{33}
\end{array}
\right)$
&4&
not small\\

$Z_3^{6}, \lambda=0$
&$\left(\begin{array}{cccc}
a_{33}&0&0\\
a_{21}&a_{22}&0\\
a_{31}&a_{32}&a_{33}
\end{array}
\right)$
&5&
not small\\
$Z_3^{7}$
&$\left(\begin{array}{cccc}
a_{33}&0&0\\
a_{32}&a_{33}&0\\
a_{31}&a_{32}&a_{33}
\end{array}
\right)$
&4&
small\\
\end{tabular}
\end{center}

where $g=-\lambda a_{22}+ \lambda a_{33}$.
\begin{proof}
Consider Theorem \ref{t2} which give the classifications of three-dimensional associative algebras. It is clear from  the equation (1)  can applied to compute all the quasi-centroid of three-dimensional algebras.
The class $Z_{3}^{1}$ has the structure constants as follows $\gamma_{13}^{2}=1, \gamma_{31}^{2}=1$ . By using condition (1), we have $a_{12}=a_{32}=0$ and $a_{33}=a_{11}.$
Therefore we obtain the quasi-centroid of $Z_{3}^{1}$ in matrix form as follows:\\
$Q\Gamma(Z)= \left\{\left(
\begin{array}{ccc}
a_{11} & 0 & a_{13} \\
a_{21} & a_{22} & a_{23} \\
0 & 0 & a_{11}
\end{array}
\right)| a_{11},  a_{21}, a_{23}\in\mathbb{C}\right\}.$
\end{proof}
\begin{cor}
The quasi-centroids of $3$-dimensional Zinbiel algebras except for the classes $Z_3^{2}, Z_3^{4},Z_3^{5}, Z_3^{6}$ are small. The dimensions of the quasi-centroids of $3$-dimensional complex Zinbiel algebras vary between 4 and 9.
\end{cor}
\begin{thm}\label{t3}
Any $4$-dimensional Zinbiel algebra Z isomorphic to one of following non-isomorphic Zinbiel algebras\\
$Z_4^1:\;e_1e_1=e_2, e_1e_2=e_3,e_2e_1=2e_3,e_1e_3=e_4, e_2e_2=3e_4, e_3e_1=3e_4 $\\
$Z_4^{2}:\;e_1e_1=e_3, e_1e_2=e_4,e_1e_3=e_4,e_3e_1=2e_4$\\
$Z_4^3:\;e_1e_1=e_3, e_1e_3=e_4,e_2e_2=e_4,e_3e_1=2e_4$\\
$Z_4^{4}:\;e_1e_2=e_3, e_1e_3=e_4,e_2e_1=-e_3$\\
$Z_4^5:\;e_1e_2=e_3, e_1e_3=e_4,e_2e_1=-e_3,e_2e_2=e_4$\\
$Z_4^{6}:\;e_1e_1=e_4, e_1e_2=e_3,e_2e_1=-e_3,e_2e_2=-2e_3+e_4$\\
$Z_4^7:\;e_1e_2=e_3, e_2e_1=e_4,e_2e_2=-e_3$\\
$Z_4^{8}:\;e_1e_1=e_3, e_1e_2=e_4,e_2e_1=-\alpha e_3,e_2e_2=-e_4$\\
$Z_4^9:\;e_1e_1=e_4, e_1e_2=\alpha e_4,e_2e_1=-\alpha e_4,e_2e_2=e_4,e_3e_3=e_4$\\
$Z_4^{10}:\;e_1e_1=e_4, e_1e_3=e_4,e_2e_1=-e_4,e_2e_2=e_4,e_3e_1=e_4$\\
$Z_4^{11}:\;e_1e_1=e_4, e_1e_2=e_4,e_2e_1=-e_4,e_3e_3=e_4$\\
$Z_4^{12}:\;e_1e_2=e_3, e_2e_1=e_4$\\
$Z_4^{13}:\;e_1e_2=e_3, e_2e_1=e_4$\\
$Z_4^{14}:\;e_1e_2=e_3, e_2e_1=e_4$\\
$Z_4^{15}:\;e_1e_2=e_3, e_2e_1=e_4$\\
$Z_4^{16}:\;e_1e_2=e_3, e_2e_1=e_4$
\end{thm}
\begin{thm}
The Quasi-centroid of $4$-dimensional  complex Zinbiel algebras have the following form:
\end{thm}
\begin{center}
\begin{tabular}{cccc}
\hline
Algebra & Q$\Gamma$(Z) & Dim &Type\;of\; Q$\Gamma$(Z)\\
\hline
$Z_4^1 $
&$\left(\begin{array}{cccc}
a_{44}&0&0&0\\
a_{21}&a_{44}&0&0\\
a_{31}&2a_{21}&a_{44}&0\\
a_{41}&3a_{31}&3a_{21}&a_{44}
\end{array}
\right)$
&10&
small\\
$Z_4^{2}$
&$\left(\begin{array}{cccc}
a_{44}&0&0&0\\
a_{21}&a_{22}&0&0\\
a_{31}&a_{32}&a_{44}&0\\
a_{41}&a_{42}&2a_{31}&a_{44}
\end{array}
\right)$
&9&
not small\\
$Z_4^3$
&$\left(\begin{array}{cccc}
a_{44}&0&0&0\\
0&a_{44}&0&0\\
a_{31}&a_{32}&a_{44}&0\\
a_{41}&a_{42}&2a_{31}&a_{44}
\end{array}
\right)$
&7&
small\\
$Z_4^{4}$
&$\left(\begin{array}{cccc}
a_{44}&0&0&0\\
0&a_{22}&0&0\\
a_{31}&a_{32}&a_{44}&0\\
a_{41}&a_{42}&0&a_{44}
\end{array}
\right)$
&6&
not small\\
$Z_4^5$
&$\left(\begin{array}{cccc}
a_{44}&0&0&0\\
0&a_{44}&0&0\\
a_{31}&a_{32}&a_{44}&0\\
a_{41}&a_{42}&0&a_{44}
\end{array}
\right)$
&6&
small\\
$Z_4^{6}$
&$\left(\begin{array}{cccc}
a_{44}&-a_{43}&0&0\\
a_{43}&h&0&0\\
a_{31}&a_{32}&h&-a_{43}\\
a_{41}&a_{42}&a_{43}&a_{44}
\end{array}
\right)$
&9&
not small\\
$Z_4^7$
&$\left(\begin{array}{cccc}
a_{33}&-a_{33}+a_{22}&0&a_{14}\\
0&a_{22}&0&a_{24}\\
a_{31}&a_{32}&a_{33}&a_{34}\\
a_{41}&a_{42}&0&a_{44}
\end{array}
\right)$
&9&
not small\\
\end{tabular}
\end{center}

\begin{center}
\begin{tabular}{cccc}
$Z_4^{8}$
&$\left(\begin{array}{cccc}
\alpha a_{21}+a_{33}&-\alpha a_{21}+a_{21}+a_{22}-a_{33}&0&0\\
a_{21}&a_{22}&0&0\\
a_{31}&a_{32}&a_{33}&0\\
a_{41}&a_{42}&0&\alpha a_{21}-a_{21}+a_{33}
\end{array}
\right)$
&7&
not small\\
$Z_4^9$
&$\left(\begin{array}{cccc}
a_{44}&a_{12}&a_{13}&0\\
0&-\alpha a_{12}+a_{44}&a_{23}&0\\
0&0&a_{44}&0\\
a_{41}&a_{42}&a_{43}&a_{44}
\end{array}
\right)$
&6&
not small\\
$Z_4^{10}$
&$\left(\begin{array}{cccc}
a_{44}&0&0&0\\
0&a_{44}&a_{23}&0\\
0&a_{32}&a_{33}&0\\
a_{41}&a_{42}&a_{43}&a_{44}
\end{array}
\right)$
&7&
not small\\
$Z_4^{11}$
&$\left(\begin{array}{cccc}
a_{44}&0&a_{13}&0\\
0&a_{22}&a_{23}&0\\
0&0&a_{44}&0\\
a_{41}&a_{42}&a_{43}&a_{44}
\end{array}
\right)$
&7&
not small\\
$Z_4^{12}$
&$\left(\begin{array}{cccc}
a_{33}&0&0&a_{14}\\
0&a_{22}&0&a_{24}\\
a_{31}&a_{32}&a_{33}&a_{34}\\
a_{41}&a_{42}&0&a_{44}
\end{array}
\right)$
&10&
not small\\
$Z_4^{13}$
&$\left(\begin{array}{cccc}
a_{33}&a_{34}&0&0\\
0&a_{44}&0&0\\
a_{31}&a_{32}&a_{33}&a_{34}\\
a_{41}&a_{42}&0&a_{44}
\end{array}
\right)$
&7&
not small\\
$Z_4^{14}$
&$\left(\begin{array}{cccc}
a_{11}&a_{12}&0&a_{14}\\
0&a_{33}&0&a_{24}\\
a_{31}&a_{32}&a_{33}&a_{34}\\
a_{41}&a_{42}&0&a_{44}
\end{array}
\right)$
&11&
not small\\
$Z_4^{15}\;, \quad \alpha \neq 0$
&$\left(\begin{array}{cccc}
a_{44}&a_{43}&0&0\\
a_{21}&a_{33}&0&0\\
a_{31}&a_{32}&a_{33}&a_{21}\\
a_{41}&a_{42}&a_{43}&a_{44}
\end{array}
\right)$
&7&
not small\\
$Z_4^{15}\;, \quad \alpha = 0$
&$\left(\begin{array}{cccc}
a_{44}&a_{43}&0&0\\
0&a_{33}&0&0\\
a_{31}&a_{32}&a_{33}&0\\
a_{41}&a_{42}&a_{43}&a_{44}
\end{array}
\right)$
&8&
not small\\
$Z_4^{16}$
&$\left(\begin{array}{cccc}
a_{44}&0&a_{13}&0\\
0&a_{22}&a_{23}&0\\
0&0&a_{44}&0\\
a_{41}&a_{42}&a_{43}&a_{44}
\end{array}
\right)$
&7&
not small
\end{tabular}
\end{center}
where $h=-2a_{43}+ a_{44}$.
\begin{proof}	
Consider $Z_4^{1}$. Applying the system of equation(1), we get
$a_{12}=a_{13}=a_{14}=a_{21}=a_{23}=a_{24}=0$. Hence, the quasi-centroids of $Z_4^{1}$ are indicated as follows\\
$$Q\Gamma(Z_4^{1})= \left\{\left(
\begin{array}{cccc}
a_{11}&0&0&0\\
0&a_{22}&0&0\\
a_{31}&a_{32}&a_{33}&a_{43}\\
a_{41}&a_{42}&a_{43}&a_{44}
\end{array}
\right)| a_{11},  a_{22}, a_{31}, a_{31},a_{33},a_{41}, a_{43}, a_{44}\in\mathbb{C}\right\}.$$
The quasi-centroid of the remaining parts of dimension four zinbiel algebras can be handled similarly, as illustrated above.
\end{proof}
\begin{cor}
The dimensions of the quasi-centroid of the $4$-dimensional associative algebras range between $7$ and $10$.
\end{cor}

\begin{cor}
\textit{(i)} In addition to the types $Z_4^{1}, Z_4^{3}, Z_4^{5}, Z_4^{9}$, any four-dimensional complex isomorphism class of the Zinbiel algebra has a small quasi-centroid.\\
\textit{(ii)}  The quasi-centroid of a $4$-dimensional complex non isomorphic of Zinbiel algebra is not small.
\end{cor}

\section{ Quasi-Derivations of Low-Dimensional Zinbiel Algebras}
This section is devoted to the description of the quasi-derivation of complex Zinbiel algebras of two, three and four dimensions.
Let $\{e_1, e_2, e_3, \cdots , e_n \}$ be  a basis of an $n$-dimensional  Zinbiel algebras,  an element $\phi$ of the quasi-derivation $\QDer(Z)$ being a linear transformation of the vector space $Z$ is represented in a matrix form $(a_{ij})_{i,j=1,2,\cdots,n},$ i.e. $d(e_i)=\sum\limits_{j=1}^{n}a_{ji}e_{j},$ $i=1,2,\cdots,n.$ According to the definition of the quasi-derivation the entries $a_{ij}$ of the matrix $(a_{ij})_{i,j=1,2,\cdots,n}$
must satisfy the following systems of equations:\quad$\sum\limits_{k=1}^{n}\Bigg(d_{ki}\gamma_{kj}^{p}+d_{kj}\gamma_{ik}^{p}-\gamma_{ij}^{k}d'_{pk}\Bigg)=0,\quad i,j,p=1,2,\cdots,n.$

\begin{thm}
The quasi-derivations of two-dimensional complex Zinbiel algebras are given as follows:
\end{thm}

\begin{center}
\begin{tabular}{ccc}
\hline
Algebra &Der&QDer\\
\hline
$Z_2^{1}$
&$\left(\begin{array}{cccc}
d_{11}&0\\
d_{21}&2d_{11}
\end{array}
\right)$
&$\left(\begin{array}{cccc}
a_{11}&0\\
a_{21}&2d_{11}
\end{array}
\right)$
\end{tabular}
\end{center}
\newpage
\begin{thm}
The quasi-derivations of three-dimensional complex Zinbiel algebras are given as follows:
\end{thm}

\begin{center}
\begin{tabular}{ccc}
\hline
Algebra &Der&QDer\\
\hline
$Z_3^{1}$
&$\left(\begin{array}{cccc}
d_{11}&d_{12}&d_{33}\\
d_{21}&d_{22}&d_{23}\\
d_{31}&d_{32}&d_{33}
\end{array}
\right)$
&$\left(\begin{array}{cccc}
a_{11}&a_{12}&a_{33}\\
a_{21}&a_{22}&a_{23}\\
a_{31}&a_{32}&a_{33}
\end{array}
\right)$
\\
$Z_3^{2}$
&$\left(\begin{array}{cccc}
d_{11}&0&0\\
d_{21}&d_{22}&0\\
d_{31}&d_{32}&2d_{11}
\end{array}
\right)$
&$\left(\begin{array}{cccc}
a_{11}&a_{12}&0\\
a_{21}&a_{22}&0\\
a_{31}&a_{32}&2d_{11}
\end{array}
\right)$
\\
$Z_3^{3}$
&$\left(\begin{array}{cccc}
d_{22}&-d_{21}&0\\
d_{21}&d_{22}&0\\
d_{31}&d_{32}&2d_{22}
\end{array}
\right)$
&$\left(\begin{array}{cccc}
a_{11}&a_{12}&0\\
a_{21}&a_{22}&0\\
a_{31}&a_{32}&2d_{22}
\end{array}
\right)$
\\
$Z_3^{4}$
&$\left(\begin{array}{cccc}
d_{33}-d_{22}&d_{12}&0\\
d_{21}&d_{22}&0\\
d_{31}&d_{32}&d_{33}
\end{array}
\right)$
&$\left(\begin{array}{cccc}
a_{11}&a_{12}&0\\
a_{21}&a_{22}&0\\
a_{31}&a_{32}&d_{11}+d_{22}
\end{array}
\right)$
\\
$Z_3^{5}$
&$\left(\begin{array}{cccc}
d_{33}-d_{22}&0&0\\
0&d_{22}&0\\
d_{31}&d_{32}&d_{33}
\end{array}
\right)$
&$\left(\begin{array}{cccc}
a_{11}&a_{12}&0\\
a_{21}&a_{22}&0\\
a_{31}&a_{32}&d_{11}+d_{22}
\end{array}
\right)$
\\
$Z_3^{6}\;, \quad \lambda  \neq 0 $
&$\left(\begin{array}{cccc}
d_{33}-d_{22}&d_{33}-2d_{22}&0\\
-d_{33}+2d_{22}&d_{22}&0\\
d_{31}&d_{32}&d_{33}
\end{array}
\right)$
&$\left(\begin{array}{cccc}
a_{11}&a_{12}&0\\
a_{21}&a_{22}&0\\
a_{31}&a_{32}&-d_{21}+2d_{22}
\end{array}
\right)$
\\
$Z_3^{6}\;, \quad \lambda=0$
&$\left(\begin{array}{cccc}
d_{33}-d_{22}&0&0\\
-d_{33}+2d_{22}&d_{22}&0\\
d_{31}&d_{32}&d_{33}
\end{array}
\right)$
&$\left(\begin{array}{cccc}
a_{11}&a_{12}&0\\
a_{21}&a_{22}&0\\
a_{31}&a_{32}&-d_{21}+2d_{22}
\end{array}
\right)$
\\
$Z_3^{7}$
&$\left(\begin{array}{cccc}
d_{11}&0&0\\
d_{21}&2d_{11}&0\\
d_{31}&\frac{3}{2}d_{21}&3d_{11}
\end{array}
\right)$
&$\left(\begin{array}{cccc}
a_{11}&0&0\\
a_{21}&2d_{11}&0\\
a_{31}&\frac{3}{2}d_{21}&d_{11}+d_{22}
\end{array}
\right)$
\end{tabular}
\end{center}
\newpage
\begin{thm}
The quasi-derivations of four-dimensional complex Zinbiel algebras are given as follows:
\end{thm}

\begin{center}
\begin{tabular}{cccc}
\hline
Algebra &Der&QDer\\
\hline
$Z_4^1$
&$\left(\begin{array}{cccc}
\frac{1}{2}d_{22}&0&0&0\\
\frac{1}{3}d_{32}&d_{22}&0&0\\
\frac{1}{4}d_{42}&d_{32}&\frac{3}{2}d_{22}&0\\
d_{41}&d_{42}&2d_{32}&2d_{22}
\end{array}
\right)$
&$\left(\begin{array}{cccc}
a_{11}&0&0&0\\
a_{21}&2d_{11}&0&0\\
a_{31}&d_{32}&d_{11}+d_{22}&0\\
a_{41}&4d_{31}&2d_{32}&2d_{22}
\end{array}
\right)$
\\
$Z_4^{2}$
&$\left(\begin{array}{cccc}
d_{11}&0&0&0\\
d_{43}-3d_{31}&2d_{11}&0&0\\
d_{31}&0&2d_{11}&0\\
d_{41}&d_{42}&d_{43}&3d_{11}
\end{array}
\right)$
&$\left(\begin{array}{cccc}
a_{11}&a_{12}&0&0\\
a_{21}&a_{22}&0&0\\
a_{31}&a_{32}&2d_{11}&0\\
a_{41}&a_{42}&d_{21}+3d_{31}&d_{11}+d_{33}
\end{array}
\right)$
\\
$Z_4^3$
&$\left(\begin{array}{cccc}
d_{11}&0&0&0\\
0&\frac{3}{2}d_{11}&0&0\\
d_{31}&0&2d_{11}&0\\
d_{41}&d_{42}&3d_{31}&3d_{11}
\end{array}
\right)$
&$\left(\begin{array}{cccc}
a_{11}&a_{12}&0&0\\
a_{21}&a_{22}&0&0\\
a_{31}&a_{32}&4d_{22}-3d_{33}&0\\
a_{41}&a_{42}&3d_{31}&2d_{22}
\end{array}
\right)$
\\
$Z_4^{4}$
&$\left(\begin{array}{cccc}
-d_{33}+d_{44}&0&0&0\\
d_{21}&-d_{44}+2d_{33}&0&0\\
0&0&d_{33}&0\\
d_{41}&d_{42}&0&d_{44}
\end{array}
\right)$
&$\left(\begin{array}{cccc}
a_{11}&a_{12}&0&0\\
a_{21}&a_{22}&0&0\\
a_{31}&a_{32}&d_{11}+d_{22}&0\\
a_{41}&a_{42}&0&d_{11}+d_{33}
\end{array}
\right)$\\
$Z_4^5$
&$\left(\begin{array}{cccc}
\frac{1}{2}d_{22}&0&0&0\\
-d_{43}&d_{22}&0&0\\
0&2d_{43}&\frac{3}{2}d_{22}&0\\
d_{41}&d_{42}&d_{43}&2d_{22}
\end{array}
\right)$
&$\left(\begin{array}{cccc}
a_{11}&a_{12}&0&0\\
a_{21}&a_{22}&0&0\\
a_{31}&a_{32}&3d_{22}-d_{33}&0\\
a_{41}&a_{42}&-d_{21}&2d_{22}
\end{array}
\right)$\\
$Z_4^{6}$
&$\left(\begin{array}{cccc}
d_{22}&0&0&0\\
0&d_{22}&0&0\\
d_{31}&d_{32}&2d_{22}&0\\
d_{41}&d_{42}&0&2d_{22}
\end{array}
\right)$
&$\left(\begin{array}{cccc}
a_{11}&a_{12}&0&0\\
a_{21}&a_{22}&0&0\\
a_{31}&a_{32}&2d_{22}&0\\
a_{41}&a_{42}&0&2d_{22}
\end{array}
\right)$
\\
$Z_4^7$
&$\left(\begin{array}{cccc}
d_{22}&0&0&0\\
0&d_{22}&0&0\\
d_{31}&d_{32}&2d_{22}&0\\
d_{41}&d_{42}&0&2d_{22}
\end{array}
\right)$
&$\left(\begin{array}{cccc}
a_{11}&a_{12}&0&0\\
a_{21}&a_{22}&0&0\\
a_{31}&a_{32}&2d_{22}&0\\
a_{41}&a_{42}&0&2d_{22}
\end{array}
\right)$\\
$Z_4^{8}\;, \quad \alpha  \neq 0 $
&$\left(\begin{array}{cccc}
d_{11}&d_{12}&0&0\\
d_{21}&d_{12}+d_{11}-d_{21}&0&0\\
d_{31}&d_{32}&2d_{11}-d_{21}&d_{12}\\
d_{41}&d_{42}&d_{21}&d_{12}+d_{11}-2d_{21}
\end{array}
\right)$
&$\left(\begin{array}{cccc}
a_{11}&a_{12}&0&0\\
a_{21}&a_{22}&0&0\\
a_{31}&a_{32}&a_{33}&d_{12}\\
a_{41}&a_{42}&d_{21}&d_{12}+a_{33}-d_{21}
\end{array}
\right)$
\\
$Z_4^{8}\;, \quad \alpha  = 0 $
&$\left(\begin{array}{cccc}
d_{11}&0&0&0\\
0&d_{11}&0&0\\
d_{31}&d_{32}&2d_{11}&0\\
d_{41}&d_{42}&0&2d_{11}
\end{array}
\right)$
&$\left(\begin{array}{cccc}
a_{11}&a_{12}&0&0\\
a_{21}&a_{22}&0&0\\
a_{31}&a_{32}&a_{33}&0\\
a_{41}&a_{42}&0&a_{33}
\end{array}
\right)$
\\
$Z_4^9 \;, \quad \alpha  \neq 0 $
&$\left(\begin{array}{cccc}
d_{22}&-d_{21}&d_{13}&0\\
d_{21}&d_{22}&d_{23}&0\\
-d_{13}&-d_{23}&d_{22}&0\\
d_{41}&d_{42}&d_{43}&2d_{22}
\end{array}
\right)$
&$\left(\begin{array}{cccc}
a_{11}&a_{12}&a_{13}&0\\
a_{21}&a_{22}&a_{23}&0\\
a_{31}&a_{32}&a_{33}&0\\
a_{41}&a_{42}&a_{43}&2d_{33}
\end{array}
\right)$
\\
$Z_4^9 \;, \quad \alpha  = 0 $
&$\left(\begin{array}{cccc}
d_{22}&-d_{21}&0&0\\
d_{21}&d_{22}&0&0\\
0&0&d_{22}&0\\
d_{41}&d_{42}&d_{43}&2d_{22}
\end{array}
\right)$
&$\left(\begin{array}{cccc}
a_{11}&a_{12}&a_{13}&0\\
a_{21}&a_{22}&a_{23}&0\\
a_{31}&a_{32}&a_{33}&0\\
a_{41}&a_{42}&a_{43}&2d_{33}
\end{array}
\right)$
\end{tabular}
\end{center}

\begin{center}
\begin{tabular}{ccc}

$Z_4^{10}$
&$\left(\begin{array}{cccc}
d_{33}&0&0&0\\
-d_{32}&d_{33}&0&0\\
0&d_{32}&d_{33}&0\\
d_{41}&d_{42}&d_{43}&2d_{33}
\end{array}
\right)$
&$\left(\begin{array}{cccc}
a_{11}&a_{12}&a_{13}&0\\
a_{21}&a_{22}&a_{23}&0\\
a_{31}&a_{32}&a_{33}&0\\
a_{41}&a_{42}&a_{43}&2d_{33}
\end{array}
\right)$
\\
$Z_4^{11}$
&$\left(\begin{array}{cccc}
d_{33}&0&0&0\\
d_{21}&d_{33}&0&0\\
0&0&d_{33}&0\\
d_{41}&d_{42}&d_{43}&2d_{33}
\end{array}
\right)$
&$\left(\begin{array}{cccc}
a_{11}&a_{12}&a_{13}&0\\
a_{21}&a_{22}&a_{23}&0\\
a_{31}&a_{32}&a_{33}&0\\
a_{41}&a_{42}&a_{43}&2d_{33}
\end{array}
\right)$
\\
$Z_4^{12}$
&$\left(\begin{array}{cccc}
d_{44}-d_{22}&0&0&0\\
0&d_{22}&0&0\\
d_{31}&d_{32}&d_{44}&0\\
d_{41}&d_{42}&0&d_{44}
\end{array}
\right)$
&$\left(\begin{array}{cccc}
a_{11}&a_{12}&0&0\\
a_{21}&a_{22}&0&0\\
a_{31}&a_{32}&d_{11}+d_{22}&0\\
a_{41}&a_{42}&0&d_{11}+d_{22}
\end{array}
\right)$
\\
$Z_4^{13}$
&$\left(\begin{array}{cccc}
d_{33}-d_{22}&d_{12}&0&0\\
0&d_{22}&0&0\\
d_{31}&d_{32}&d_{33}&0\\
d_{41}&d_{42}&0&2d_{22}
\end{array}
\right)$
&$\left(\begin{array}{cccc}
a_{11}&a_{12}&0&0\\
a_{21}&a_{22}&0&0\\
a_{31}&a_{32}&d_{11}+d_{22}&0\\
a_{41}&a_{42}&0&2d_{22}
\end{array}
\right)$
\\
$Z_4^{14}$
&$\left(\begin{array}{cccc}
d_{44}-d_{22}&d_{43}&0&0\\
0&d_{22}&0&0\\
d_{31}&d_{32}&2d_{22}&0\\
d_{41}&d_{42}&d_{43}&d_{44}
\end{array}
\right)$
&$\left(\begin{array}{cccc}
a_{11}&a_{12}&0&0\\
a_{21}&a_{22}&0&0\\
a_{31}&a_{32}&2d_{22}&0\\
a_{41}&a_{42}&d_{12}&d_{11}+d_{22}
\end{array}
\right)$
\end{tabular}
\end{center}

\begin{center}
\begin{tabular}{ccc}
$Z_4^{15}\;, \quad \alpha  \neq 1$
&$\left(\begin{array}{cccc}
d_{44}-d_{22}&-\frac{1}{2}d_{43}\alpha+\frac{1}{2}d_{43}&0&0\\
0&d_{22}&0&0\\
d_{31}&d_{32}&2d_{22}&0\\
d_{41}&d_{42}&d_{43}&d_{44}
\end{array}
\right)$
&$\left(\begin{array}{cccc}
a_{11}&a_{12}&0&0\\
a_{21}&a_{22}&0&0\\
a_{31}&a_{32}&2d_{22}&0\\
a_{41}&a_{42}&a_{43}&d_{11}+d_{22}
\end{array}
\right)$
\\
$Z_4^{16}$
&$\left(\begin{array}{cccc}
2d_{33}-d_{22}&d_{12}&0&0\\
d_{21}&d_{22}&0&0\\
0&0&d_{33}&0\\
d_{41}&d_{42}&d_{43}&2d_{33}
\end{array}
\right)$
&$\left(\begin{array}{cccc}
a_{11}&a_{12}&a_{13}&0\\
a_{21}&a_{22}&a_{23}&0\\
a_{31}&a_{32}&a_{33}&0\\
a_{41}&a_{42}&a_{43}&2d_{33}
\end{array}
\right)$
\end{tabular}
\end{center}

\end{document}